\documentclass[12pt,a4paper]{article}
\usepackage[utf8]{inputenc}
\usepackage[english]{babel}
\usepackage{indentfirst}
\usepackage{misccorr}
\usepackage{graphicx}
\usepackage{amsmath,amssymb,url,xcolor}
\usepackage{amsthm}
\usepackage{scrextend}
\usepackage[T2A]{fontenc}
\usepackage{mathtools}
\usepackage[colorlinks=true,linkcolor=blue,urlcolor=red,unicode=true,hyperfootnotes=false,bookmarksnumbered]{hyperref}

\newcommand{\ff}{\mathcal{F}}
\newcommand{\cG}{\mathcal{G}}
\newcommand{\cA}{\mathcal{A}}
\newcommand{\bb}{\mathcal{B}}

\newtheorem{Theorem}{Theorem}
\newtheorem{Claim}[Theorem]{Claim}
\newtheorem{Lemma}[Theorem]{Lemma}

\title{On supersaturation in the Erd\H{o}s--S\'os problem}
\author{Andrey Kupavskii\thanks{Moscow Institute of Physics and Technology, Saint-Petersburg State University, Innopolis University;
E-mail: \url{kupavskii@ya.ru}}, Yakov Shubin\thanks{Moscow Institute of Physics and Technology;
E-mail: \url{shubin.yakoff@gmail.com}}}

\begin{document}
\maketitle
\begin{abstract}
    The following classical question in extremal set theory is due to Erd\H os and S\'os: what is the size of the largest family $\ff\subset {[n]\choose k}$ with no two sets $F_1,F_2\in \ff$ such that $|F_1\cap F_2| = t$? In this paper, we address a supersaturation question for this extremal function. For a family $\ff\subset {[n]\choose k}$ of a fixed size $\ell$, what is the smallest number of pairs $F_1,F_2\in \ff$ with $|F_1\cap F_2|=t$ it may induce? For  fixed $k$ and $n\to \infty$, we find the exact threshold when the minimum number of pairs matches the expected number of pairs in a random $\ell$-element family up to a constant factor. We also find an exact answer for $\ell$ slightly above the extremal function. 
\end{abstract}
\section{Introduction}
For a positive integer $n$, denote $[n] = \{1,2,\ldots, n\}$. Given a set $X$ and an integer $k \geqslant 0$, denote by $2^X$ and $\binom{X}{k}$ the collection of all subsets and all $k$-element subsets ($k$-sets) of $X$, respectively. A {\it family} is simply a collection of sets. 

One of the classical topics in extremal set theory is the study of systems with restricted intersections. It was initiated with the paper of Erd\H os, Ko and Rado \cite{EKR}. For all $n,k$ they determined the size of the largest family $\ff\subset {[n]\choose k}$ that is {\it intersecting}, i.e., for any two sets $A,B\in \ff$ we have $A\cap B=\emptyset$.  We may also say that such $\ff$ {\it avoids intersections of size $0$}. Erd\H os, Ko and Rado also determined for $n>n_0(k,t)$ the largest $t$-intersecting families in ${[n]\choose k}$, i.e., families $\ff$ with the condition $|A\cap B|\ge t$ for any $A,B\in\ff$. They also posed several questions in that regard. Answering one of those, Katona~\cite{Kat} determined the size of the largest $t$-intersecting subfamily of $2^{[n]},$ for all $n,t.$ The uniform case of the problem, however, turned out to be much harder. Frankl \cite{F1} and Wilson \cite{W} determined the size of the largest $t$-intersecting subfamily of ${[n]\choose k}$ for $n\ge (t+1)(k-t+1)$, which is exactly the range when the answer to the question is ${n-t\choose k-t}.$ After a partial result of Frankl and F\"uredi \cite{FF0}, Ahlswede and Khachatrian~\cite{AK} confirmed the conjecture of Frankl~\cite{F1} and found the exact solution to the $t$-intersecting problem for all $n,k,t.$

Erd\H{o}s and S\'os~\cite{ES} suggested a variant with a weaker condition: what is the largest family $\mathcal{F} \subset\binom{[n]}{k}$ that {\it avoids intersection $t$}, i.e. such that $|A \cap B| \neq t$ for any $A, B \in \mathcal{F}$? More generally, Deza, Erd\H os and Frankl~\cite{DEF} studied $(n,k,L)$-systems, i.e., families  $\ff\subset {[n]\choose k}$ such that for any two sets $F_1,F_2\in \ff$ we have $|F_1\cap F_2|\in L.$ They introduced the Delta-system method, which was later developed largely in connection with the study of $(n,k,L)$-systems. We refer to a recent survey of the first author \cite{Kupdel} on the Delta-system method. 
We should also mention the results of Ray-Chaudhuri and Wilson~\cite{RW}, as well as the result of Frankl and Wilson~\cite{FW}, which give bounds for the sizes of $(n,k,L)$-systems via the linear-algebraic method.

Returning to the Erd\H os--S\'os problem, the answer to this question is not known in general. One of the achievements of the Delta-system method was an almost complete resolution of this question for $n>n_0(k)$ by Frankl and Furedi~\cite{FF}. Their methods play a very important role in this paper. It would be convenient for us to formulate their result in graph terms. Let $G(n,k,t)$ be the Generalized Johnson graph with $V(G(n,k,t)) = \binom{[n]}{k}$ and $$E(G(n,k,t)) = \left\{ \{F_1, F_2\}: |F_1 \cap F_2| =t, F_1, F_2 \in \binom{[n]}{k} \right\}.$$
In these terms, the Erd\H os--S\'os problem asks to determine the independence number $\alpha(G(n,k,t))$ of $G(n,k,t).$ 
\begin{Theorem}[Frankl and F\"uredi, \cite{FF}]\label{thmff}
(i) Fix positive integers $k,t$ with $k > 2t+1$. There exists $n_0$ such that for $n > n_0$
$$\alpha(G(n,k,t)) = \binom{n-t-1}{k-t-1} \sim \cfrac{n^{k-t-1}}{(k-t-1)!}.$$
(ii) Fix positive integers $k,t$ with $k \leqslant 2t+1$. Then
$$\alpha(G(n,k,t)) = \Theta (n^t).$$
\end{Theorem}
Theorem~\ref{thmff} has two regimes. The extremal example in the first case is the same as in the Frankl~\cite{F1} and Wilson~\cite{W} result: the family of all sets containing some fixed $t+1$ elements. That is, the extremal family is $(t+1)$-intersecting. In the second regime, the example is supposedly design-like (or Steiner system-like). E.g., the largest family of $k$-sets with pairwise intersections at most $t-1$ has size $\Omega(n^t)$ and avoids intersection $t$. The actual supposedly extremal example is a bit more complicated, but is still based on a Steiner system. Frankl and F\"uredi in \cite{FF} verified that this example is indeed extremal in some cases via linear algebra.   

More recently, new methods were introduced to the study of the Erd\H os--S\'os problem, including the junta method, hypercontractivity and spread approximations. See the recent papers by Keller and Lifshitz~\cite{KL},  Ellis, Keller and Lifshitz~\cite{EKL}, Kupavskii and Zakharov~\cite{KZ}, Kupavskii and Noskov~\cite{KN}. All of them address the first  regime, resolving the problem under different dependencies between the parameters $n,k,t$.

A popular theme in extremal combinatorics are supersaturation questions, which could be vaguely formulated as follows: once the object under study is past the extremal threshold, how many copies of the forbidden structure would it contain? One notable example is the asymptotically sharp supersaturation result of Razborov \cite{Raz}, then upgraded to an exact result by Liu, Pikhurko, and Staden \cite{LPS}, which gives the minimum number of triangles in an $n$-vertex graph with a fixed number of edges.

This paper is devoted to supersaturation for the Erd\H os--S\'os problem, which  may be conveniently formulated in terms of the Generalized Johnson graph: \begin{quote}
    for a fixed number $\ell$, what is the minimum number $\rho(\ell)$ of edges in an induced subgraph of $G(n,k,t)$ on $\ell$ vertices?
    \end{quote}
Note that, by definition, $\rho(\ell)=0$ for $\ell\le \alpha(G(n,k,t)).$ We may formulate the question directly in terms of set families. For a family $\mathcal{F} \subset\binom{[n]}{k}$, let $\rho(\mathcal{F})$ be the number of pairs $F_1, F_2$ such that $F_1, F_2 \in \mathcal{F}$ and $|F_1 \cap F_2| =t.$ Then
$$\rho(\ell) = \min_{|\mathcal{F}|=\ell, \mathcal{F} \subset\binom{[n]}{k}} \rho(\mathcal{F}).$$

One case of this problem that is particularly well-studied is supersaturation in Kneser graphs, i.e., graphs $G(n,k,0),$ as well as its analogue for $2^{[n]}.$ The Erd\H os--Ko--Rado Theorem states that, for $n\ge 2k$, $\alpha(G(n,k,0))={n-1\choose k-1}.$ Several groups of researchers showed that lexicographic families have the smallest  or asymptotically smallest number of disjoint pairs in different regimes of the parameters (the three parameters being $n,k,s$, where $s\sim |\ff|/{n-1\choose k-1}$): see Poljak and Tuza \cite{PT}, Das, Gan and Sudakov \cite{DGS}, Frankl, Kohayakawa and R\"odl \cite{FKR}, Balogh, Das Liu, Sharifzadeh, and Tran \cite{BDLST}. The question itself was posed by Ahlswede \cite{Ahl}. The non-uniform variant of this question was independently resolved by Frankl \cite{Fr17} and by Ahlswede \cite{Ahl}. Bollob\'as and Leader \cite{BL} reproved the result of \cite{Fr17, Ahl} and also proved an exact result concerning $t$-disjoint pairs (pairs with intersection $< t$), albeit for sizes of the form $\sum_{i=0}^k{n\choose i}$ only. An example in the paper by Frankl \cite{Fr17} shows that the answer to the question for arbitrary sizes can be more complicated than what one may expect. Supersaturation for $t$-disjoint pairs in ${[n]\choose k}$ was treated in \cite{DGS} and then in \cite{FKR}.

Let us briefly comment on the differences between the problem of determining $\rho(\ell)$ and its `monotone' counterpart of determining the minimum number of $t$-disjoint pairs for a family $\ff\subset {[n]\choose k}$ of size $\ell.$

Frankl, Kohayakawa and R\"odl \cite{FKR} found an elegant and rather simple way of getting approximate results on supersaturation for disjoint pairs or $t$-disjoint pairs in ${[n]\choose k}$ (i.e., edges in induced subgraphs of the graph $G(n,k,< t)$) for moderately large $n$. Let us treat the case of disjoint pairs and a family $\ff$ of size $s{n-1\choose k-1}$. Take a random perfect matching $\mathcal M$ in ${[n]\choose k}$. Its expected intersection with $|\ff|$ has size $s$, and thus gives rise to at least $s^2/2$ disjoint pairs on average (here, we are using Jensen's inequality for the function ${x\choose 2}$). At the same time, the probability that an edge `survives' when intersected with a random matching is the probability that both endpoints survive, which is roughly $|\mathcal M|^2/{n\choose k}^2 = {n-1\choose k-1}^{-2}.$ By linearity of expectation, we conclude that $\ff$ induces at least $\frac{s^2}2{n-1\choose k-1}^2$ edges. For integer $s$ and $n\gg k^2,$ this is the correct asymptotics. A similar argument, but with a Steiner system replacing a perfect matching works for $t$-disjoint pairs. There are also other tools available for Kneser graphs and graphs $G(n,k,< t)$, such as Hoffman/expander mixing-type bounds. They also give asymptotically tight bounds on the number of edges in vertex sets beyond the independence number.

The situation for pairs with intersection exactly $t$, i.e., for graphs $G(n,k,t)$ is much more complicated, in particular because neither Steiner system-type nor spectral-type argument is available. The behavior of the function $\rho(\ell)$ is also very different. For $t$-disjoint pairs and families of size $\gg {n-t\choose k-t}$ {\it most} pairs were $t$-disjoint, but this is far from being the case for the pairs with intersection exactly $t$ (and families of size $\gg {n-t-1\choose k-t-1}$). Lastly, there is of course the difficulty that we do not even know the independence number of $G(n,k,t)$ in general. Indeed, some previous works (e.g., \cite{YS}, \cite{DNRS}, \cite{Push}) treated the problem of determining $\rho(\ell),$ but could not determine the correct order of magnitude in general. For the most part, these works studied the case $k=3, t=1$.  Even in this specific setting, the exact multiplicative constant for $\ell = \Theta(n^2)$ remained unknown. For  any constant $k,t$ and $n\to \infty$  the correct order of magnitude was determined for the case $\ell \gg n^{k-2}$ \cite{YS2}. In the paper \cite{DNRS2}, the authors studied a related problem of the minimum number of induced cliques in subgraphs of $G(n,k,t)$.

In this paper, we managed to determine the order of magnitude of the function $\rho(\ell)$ for almost any regime, provided $k,t$ are fixed and $n\to\infty.$ In some cases, we even determined the exact value of $\rho(\ell).$

\subsection{Results}

The main results of this paper are the lower bounds on $\rho(\ell)$. We, however, start with the upper bounds in order to put the lower bounds into context. 

The following easy upper bound follows from averaging: it gives the average number of edges induced on a random subset of vertices of $G(n,k,t)$.
\begin{Claim}\label{cla2}
For each $\ell > \alpha(G(n,k,t))$ we have
$$
\rho(\ell) \leqslant(1+o(1)) \cdot \frac{\ell^{2}}{n^{t}} \cdot \frac{t!}{2} \cdot \binom{k}{t}^2.
$$
\end{Claim}

\begin{proof}
    For convenience, denote $G=G(n,k,t)$ and $V = V(G)$, $E = E(G)$. Note that $|V| = \binom{n}{k}$ and 
    $$|E| = \frac{|V|\binom{k}{t}\binom{n-k}{k-t}}{2} =(1+o(1)) \frac{|V|^2\binom{k}{t}\prod_{i=0}^{k-t+1}(k-i)}{2 n^t}=(1+o(1)) \frac{|V|^2\binom{k}{t}^2}{2 n^t t!}.$$
    Let us choose a uniformly random induced subgraph $H$ of $G(n,k,t)$ on $\ell$ vertices. The probability of any given edge to fall into $H$ is ${|V|-2\choose \ell-2}/{|V|\choose \ell}$. By linearity of expectation, the expected number of edges in $H$ is thus equal to 
    $$|E|\frac{\binom{|V|-2}{l-2}}{\binom{|V|}{\ell}}=|E|\frac{\ell(\ell-1)}{|V|(|V|-1)} =(1+o(1)) \frac{\ell^{2}}{n^{t}} \cdot \frac{t!}{2} \cdot \binom{k}{t}^2.$$
    Take any subgraph which induces at most that many edges. The claim follows.
\end{proof}

In the case $n^{k-1} = o(\ell)$ Claim 2 gives us an asymptotically tight bound, see \cite{YS}. However, in most other regimes we can improve upon this bound. For instance, the following bound was shown in \cite{YS} via an explicit construction.

\begin{Claim}\label{cla3}
Fix positive integers $k,t$ such that $k \geqslant 2t+1$. Let $n^{k-t-1} = o(\ell)$ and $\ell = o (n^{k-t})$. Then
$$
\rho(\ell) \leqslant(1+o(1)) \cdot \frac{\ell^{2}}{n^{t}} \cdot \frac{t!}{2} \cdot \binom{k-t-1}{t}^2.
$$
\end{Claim}

We now turn to the upper bounds. We know the order of magnitude of $\alpha = \alpha(G(n,k,t))$, and thus we can apply Turan's Theorem \cite{TT}. It gives us the bound
$\rho(\ell) \geqslant \frac{\alpha}{2}\left[\frac{\ell}{\alpha}\right](\left[\frac{\ell}{\alpha}\right]-1).$ It is easy to see that in the case $k \leqslant 2t+1$ the lower bound on $\rho(\ell)$ from Turan's Theorem is of the form  $\Theta(\frac{\ell^2}{n^t})$, which matches  the upper bounds from Claims~\ref{cla2} and~\ref{cla3} up to a constant factor. In the case $k > 2t+1$, however, the exponent of $n$ does not match.

In Section~\ref{sec3}, we prove the following result.

\begin{Theorem}\label{thm4}
Fix positive integers $k$, $t$ such that $k \geqslant 2t+1$ and a function $\ell(n)$ satisfying $n^{k-t-1} = o(\ell)$. Then
$$\rho(\ell) \geqslant (1+o(1)) \cdot \frac{\ell^{2}}{n^{t}} \cdot \frac{t!}{2}.$$
\end{Theorem}

Theorem~\ref{thm4} gives us the correct order of magnitude. Moreover, in the case $k=2t+1$ it asymptotically matches the bound in Claim~\ref{cla3}.

Theorem~\ref{thm4}, however, does not cover the case $\ell = \Theta(n^{k-t-1}).$ For the next two results, it is convenient for us to parametrize $\ell(n) = \binom{n-t-1}{k-t-1} + r(n),$ where $r=r(n) > 0$ and $r= O(n^{k-t-1})$.

\begin{Theorem}\label{thm5}  Fix positive integers $k,t$ with $k \geqslant 2t+3$ and a function $\ell(n) = \binom{n-t-1}{k-t-1} + r(n)$ such that $r=r(n)$ is a positive integer-valued function.
\begin{itemize}
    \item[(i)]If $r=o(n^{k-t-1})$, then 
    $$\rho(\ell) = \Theta (r n^{k-2t-1}).$$
    \item[(ii)] If $r=\Theta(n^{k-t-1})$, then
    $$\rho(\ell) = \Theta (n^{2k-3t-2}).$$
\end{itemize}
\end{Theorem}
Speaking of the supersaturation phenomenon, we see that in Theorem~\ref{thm4} and Theorem~\ref{thm5}~(i), we have $\rho(\ell) = \Theta(\frac{\ell^2}{n^{t}})$, but in the case $r=o(\alpha(G(n,k,t)))$ of Theorem~\ref{thm5}~(ii) we observe that $\rho(\ell)=o(\frac{\ell^2}{n^t})$. That is, the minimum number of edges in a subgraph of $G(n,k,t)$ of size $\ell$ and the average number of edges in a random subgraph of the same size are the same up to a constant factor iff $\ell>(1+\epsilon)\alpha(G(n,k,t)).$

The final theorem of this section determines the exact minimum number of edges for small  $r(n)$.

\begin{Theorem}\label{thm6}
    Fix positive integers $k,t$ with $k \geqslant 2t+3$ and a function $\ell(n) = \binom{n-t-1}{k-t-1} + r(n)$ such that $r=r(n)$ is a positive integer-valued function. If $r=o(n^{k-2t-1})$, then there exists $n_0$ such that for $n > n_0$
    $$\rho(\ell) = r \binom{k}{t}\binom{n-k-t-1}{k-2t-1}.$$
\end{Theorem}

We organize the rest of the paper as follows. In the next section, we introduce the notation, several important tools for the proofs and sketchs of the proofs. In Section~\ref{sec3}, we prove Theorem~\ref{thm4}. In Subsections~\ref{sec41} and~\ref{sec42}, we prove the lower bound in Theorem~\ref{thm5}~(i) and give two different extremal configurations for the upper bound. In Section~\ref{sec43}, we adapt the proof of Theorem~\ref{thm5}~(i) to derive the lower bound in Theorem~\ref{thm5}~(ii). In Section~\ref{sec5}, we prove Theorem~\ref{thm6} and introduce one more construction that demonstrates the optimality of the bound on $r(n)$ in Theorem~\ref{thm6}. 

\section{Preliminaries}\label{sec2}

We start with several classical definitions.

A family $\mathcal{F} = \{F_1, \ldots, F_s\}$ is a \textit{sunflower} with $s$ petals and kernel $D$ if $F_i \cap F_j = D$ holds for any distinct $F_i$ and $F_j$ from $\mathcal{F}$. Sunflowers are also called \textit{$\Delta$-systems}; therefore, it can be said that we will employ the $\Delta$-system method in our proofs.

A family $\mathcal{F} \subset\binom{[n]}{k}$ is \textit{$k$-partite} if there exists a partition $[n]=X_1 \sqcup X_2 \sqcup \ldots \sqcup X_k$ such that for any $F \in \mathcal{F}$ and $i \in[k]$ we have $\left|F \cap X_i\right|=1$. Further, for each set $B \in \mathcal{F}$ define the trace $\left.\mathcal{F}\right|_B:=\{A \cap B: A \in$ $\mathcal{F}, A \neq B\}$. We exclude $B$ itself from the trace for convenience. If $\mathcal{F}$ is $k$-partite with parts $X_1, \ldots, X_k$, then each set in $\left.\mathcal{F}\right|_B$ intersects each part $X_i$ in at most 1 element. The following notion encodes, which parts are intersected. Define the projection $\pi: 2^{[n]} \rightarrow[k]$ as follows: $\pi(Y)=\left\{i \in[k]: Y \cap X_i \neq \emptyset\right\}$. We can naturally extend this definition to $\pi\left(\left.\mathcal{F}\right|_B\right) \subset 2^{[k]}$. Define \textit{intersection structure} $\operatorname{Int}(\mathcal{F})$ as follows: 
$$
\operatorname{Int}(\mathcal{F})=\{\pi(E \cap F): E, F \in \mathcal{F}, E \neq F\}.
$$
Also, define
$$F_J:=F \cap\left(\bigcup_{i \in J} X_i\right),$$
where $J \subset [k].$

In what follows, we state F\"uredi’s Structural Theorem \cite{F}, a key ingredient of the $\Delta$-system method.

\begin{Theorem}\label{thm7}
(F\"uredi). Fix some positive integers $s,k$. Let $\mathcal{F}$ be a family of $k$-element sets. Then there exists $c(k, s)>0$ and a $k$-partite family $\mathcal{F}^* \subset \mathcal{F}$ such that\\
(1) $\left|\mathcal{F}^*\right| \geqslant c(k, s)|\mathcal{F}|$;\\
(2) $\pi\left(\left.\mathcal{F}^*\right|_A\right)=\pi\left(\left.\mathcal{F}^*\right|_B\right)$ for any $A, B \in \mathcal{F}$;\\
(3) for each $A, B \in \mathcal{F}^*$ their intersection $A \cap B$ is a kernel of a sunflower with $s$ petals contained in $\mathcal{F}^*$.
\end{Theorem}

We will also use an alternative, more quantitative version of this theorem due to Jiang and Longbrake \cite{JL}.

For a family $\mathcal{F}$ and set $X$, we denote 
\begin{align*}
    \mathcal{F}(X) = \{F \setminus X : F \in \mathcal{F}, X \subset F \}\\
    \mathcal{F}[X] = \{F: F \in \mathcal{F}, X \subset F \}.    
\end{align*}

Given a positive real number $s$, a family $\mathcal{H} \subset 2^{[n]}$ is called \textit{$s$-diverse} if

$$
\forall v \in [n], d_{\mathcal{H}}(v) \leqslant (1 / s)|\mathcal{H}|,
$$
where $d_{\mathcal{H}}(v)$ is the number of $H \in \mathcal{H}$ such that $v \in H$. 

\begin{Theorem}\label{thm8}
Fix some integers $s,k \geqslant 2$ such that $s \geqslant 2k$. Let $\mathcal{F}$ be a family of $k$-element sets. Then there exists $c(k, s)>0$ and a $k$-partite family $\mathcal{F}^* \subset \mathcal{F}$ such that\\
(1) $\left|\mathcal{F}^*\right| \geqslant c(k, s)|\mathcal{F}|$;\\
(2) For each $J \in \operatorname{Int}(\mathcal{F}^*)$ and $F \in \mathcal{F}^*$, $$|\mathcal{F}^*(F_J)| \geqslant \max\left\{s, \cfrac{|\mathcal{F}^*|}{2k \cdot n^{|J|}}\right\}.$$ 
(3) For each $J \in \operatorname{Int}(\mathcal{F}^*)$ and $F \in \mathcal{F}^*$, $\mathcal{F}^*(F_J)$ is $s$-diverse.
\end{Theorem}

A key property of $\mathcal{F}^*$ from both theorems is that $\operatorname{Int}(\mathcal{F}^*)$ is closed under intersections. It means that for any $A,B \in \operatorname{Int}(\mathcal{F}^*)$ set $A \cap B$ also belongs to $\operatorname{Int}(\mathcal{F}^*)$. 

Another ingredient, we need is the famous Kruskal-Katona Theorem \cite{Kru}, \cite{Kat2}. For an integer $i \geqslant 0$, we define the \textit{$i$-shadow} of $\mathcal{F}$ to be

$$
\partial^{(i)}(\mathcal{F}):=\{D:|D|=i, \exists F \in \mathcal{F}, D \subset F\}
$$

The Lovász' version (\cite{Lov}, 13.31b) of the Kruskal-Katona Theorem states that if $\mathcal{F}$ is a $k$-graph of size $|\mathcal{F}|=\binom{x}{k}$, where $x \geqslant k$ is a real number, then for all $i$ with $1 \leqslant i \leqslant k-1$ one has

$$
\left|\partial^{(i)}(\mathcal{F})\right| \geqslant \binom{ x}{i}
$$

Lastly, we will need one lemma and one theorem from \cite{FF}.

We say that some family $\mathcal{M}$ \textit{covers} a set $A$, if there exists $M \in \mathcal{M}$ such that $A \subset M$.
Recall that the \textit{rank}  $r(\mathcal{M})$ of $\mathcal{M} \subset 2^{[k]}$ is the size of the smallest set in $2^{[k]}$ not contained in any set from $\mathcal{M} \backslash[k]$ (not covered by $\mathcal{M}$).

\begin{Lemma} \label{lem9}
(\cite{FF}, Lemma 5.5).
Fix some integers $k,t$ such that $k \geqslant 2t+3.$  Let $\mathcal{M} \subset 2^{[k]}$ be the family that is closed under intersection. Suppose that $\mathcal{M}$ covers all $(k-t-2)$-element subsets of $[k]$, i.e. $r(\mathcal{M}) \geqslant k-t-1$. Then (a) or (b) holds.\\
(a) There exists some $M \in \mathcal{M}$ with $|M|=t.$\\
(b) There exists a unique set $M \in \mathcal{M}$ with $|M|=t+1$ such that $\mathcal{M}$ consists of all supersets of $M$ and some at most $(t-1)$-elements subsets.
\end{Lemma}

\begin{Theorem} \label{thm10}
(\cite{FF}, Theorem 5.1).
Fix some integers $k,t$ such that $k \geqslant 2t+1.$ Suppose $\mathcal{F} \subset \binom{[n]}{k}$ is $t$-avoiding, i.e. for all distinct $F, G \in \mathcal{F}$, $|F \cap G| \neq t$ holds. Then there exists a constant $c=c(k)$ such that
$$|\partial^{(k-t-1)}(\mathcal{F})| \geqslant c|\mathcal{F}|.$$
\end{Theorem}

Lemma~\ref{lem9} gives us information about structure of families with rank at least $k-t-1$. We call the set $M$ as in (b) the \textit{center} of $\mathcal{M}$.

\begin{Claim}\label{cla11}
    Let $\ff \subset \binom{[n]}{k}$ be a $k$-partite family. Suppose that $|\ff| > \binom{n}{k-t-2}$. Then $r(\operatorname{Int}(\ff)) \geqslant k-t-1$.
\end{Claim}

\begin{proof}
    Arguing indirectly, assume that $r(\operatorname{Int}(\ff)) < k-t-1$. Then there exists $J \subset 2^{[k]}$ such that $|J|=k-t-2$ and $J$ is uncovered by $\operatorname{Int}(\ff)$. Then all $F_J$ are distinct for $F \in \ff$, and so $|\mathcal{F}| \leqslant \binom{n}{k-t-2}$, contradiction.    
\end{proof}

We will combine Claim~\ref{cla11} with Lemma~\ref{lem9} in our proof of Theorem~\ref{thm5}.

\subsection{Sketch of the proofs}


We begin the proof of Theorem~\ref{thm4} by an iterative application of Theorem~\ref{thm7} in the spirit of the approach of Frankl and F\"uredi \cite{FF}. This allows us to decompose almost all of our family into several disjoint highly structured subfamilies as in Theorem~\ref{thm7}. These subfamilies will be sufficiently large to guarantee that they are not $t$-avoiding. The conditions of Theorem~\ref{thm7} will imply that all sets within these subfamilies have intersections of size $t$. Each intersection is, in fact, a center of a sunflower. Using this fact, we are able to extract a small number of sets that correspond to an almost-dominating set within our graph 
$G(n,k,t)$. We then remove this set and repeat the entire procedure, including the decomposition step. This resembles one of the proofs of Tur\' an's Theorem, in which we iteratively remove small dominating  sets in the graph to obtain a lower bound on the number of edges. 

The proofs of Theorems~\ref{thm5} and~\ref{thm6} have a common proof strategy. First, we prove a harder Theorem~\ref{thm5}~(i), and then modify/simplify its proof to obtain Theorem~\ref{thm5}~(ii) and Theorem~\ref{thm6}. The approach combines decomposition, bootstrapping and peeling arguments in the spirit of Frankl and F\"uredi \cite{FF}. We were also inspired by the paper of Jiang and Longbrake~\cite{JL} who used the Delta-system method for a supersaturation question.
Let us outline the proof of Theorem~\ref{thm5} (i) step by step.  \\
\begin{enumerate}
    \item Using Theorem~\ref{thm8} iteratively, decompose almost all the original family into disjoint regularized subfamilies. Applying Claim~\ref{cla11} and Lemma~\ref{lem9}, we will show that the intersection structure of these subfamilies must necessarily satisfy property (a) or property (b) of Lemma~\ref{lem9}.
    \item Assume that one of these subfamilies is not $t$-avoiding, i.e., if the corresponding  intersection structure satisfies property (a). We then find enough intersections of size $t$ inside that family. In what follows, we assume that all intersection structures of our subfamilies satisfy property (b).
    \item As in \cite{FF} and \cite{JL}, by counting $t$-shadows, we can find a small number of $(t+1)$-stars, which cover a bulk of our family. Recall that family $\ff$ is an $s$-star if there exists a set $S$ such that $|S|=s$ and  $\forall F \in \ff: S \subset F$.
    \item Next, we apply a certain peeling/cleaning procedure to the obtained $(t+1)$-stars. This will allow us to state that, in the resulting family,  almost each common $t$-shadow of the $(t+1)$-star families yields a sufficient  number of $t$-intersections. Via the Kruskal--Katona theorem, this either leads to sufficiently many $t$-intersections, or to the conclusion that one of the $(t+1)$-stars forms a bulk of our family.
    \item We arrive at a situation when one  $(t+1)$-star $\cA$ contains most of the sets, and we have a remainder family $\bb$.  We find sufficiently many edges either between $\mathcal{A}$ and $\mathcal{B}$, or inside the family $\mathcal{B}$ itself. Interestingly, both scenarios in some regimes have a corresponding example with the correct order of the number of edges. We give those in Section~\ref{sec42}.
\end{enumerate}

\section{Proof of Theorem 4}\label{sec3}

Take $\ff \subset \binom{[n]}{k}$ with $|\ff| = \ell.$ Consider the subgraph of $G(n,k,t)$ induced on the family $\mathcal{F}$, and bound its number of edges. For any $\mathcal{S} \subset \ff$ let $G[\mathcal{S}]$ be the subgraph induced on the family $\mathcal{S}$.

Fix an integer $s > 2k$ . Apply Theorem~\ref{thm7} to $\mathcal{F}$ with parameters $s$ and $k$ and obtain a subfamily $\mathcal{F}_1 \subset \mathcal{F}$. Put $\ff_1' :=\ff \setminus \ff_1$ and in the same way find $\mathcal{F}_2 \subset \ff_1'$, and then iteratively $\ff_3, \ldots$, until the remaining family $\ff_i'$ contains at most $\frac{\alpha(G(n,k,t))}{c(k,s)}$ elements, where $c(k,s)$ is the constant from Theorem~\ref{thm7}. Let $m$ be the number of obtained families, and let $\mathcal{F}_{m+1} := \ff_m'$ be the family of remaining sets.

This process yields a decomposition of the original family:
\[
\mathcal{F} = \mathcal{F}_1 \sqcup \mathcal{F}_2 \sqcup \ldots \sqcup \mathcal{F}_m \sqcup \mathcal{F}_{m+1}.
\]
Denote $\mathcal{F}^* = \mathcal{F} \setminus \mathcal{F}_{m+1}$. Our goal is to find a small dominating set in the subgraph induced on $\mathcal{F}^*$. 

Note that, for any $i \in [m]$, the following holds.
\[
|\mathcal{F}_i| > c(k,s) \cdot \frac{\alpha(G(n,k,t))}{c(k,s)} = \alpha(G(n,k,t)).
\]
It implies that a subgraph $G[\ff_i]$ contains at least one edge (i.e., an intersection of size $t$). By condition (2) of Theorem~\ref{thm7}, each set in $\mathcal{F}_i$ is incident to at least one edge. Thus, for every $i \in [m]$ and $F \in \mathcal{F}_i$, there exists a set $A \in \mathcal{F}_i$ such that $|F \cap A| = t$.

For each $t$-element intersection $X$ of two sets in $\mathcal{F}_i$, Theorem~\ref{thm7} guarantees the existence of a sunflower with kernel $X$ and $s$ petals. Select one such sunflower for every $X \in \binom{[n]}{t}$ that is an intersection of two sets from $\ff_i$ for some $i \in [m]$. Let $\mathcal{S}_1 \subset \mathcal{F}^*$ be the family consisting of the selected sunflowers. The number of kernels is at most $\binom{n}{t}$, and the number of petals of each selected sunflower is $s$, so
\[
|\mathcal{S}_1| \leqslant s \cdot \binom{n}{t}.
\]
Every set in $\mathcal{S}_1$ is a petal of at least one sunflower with kernel of size $t$, so as a vertex of $G$, it has at least $s-1$ adjacent vertices within $\mathcal{S}_1$. Moreover, each set in $\mathcal{F}^* \setminus \mathcal{S}_1$ has at least $s - (k-t)$ edges to $\mathcal{S}_1$. Indeed, any set $F \in \mathcal{F}_i \subset \mathcal{F}^*$ has at least one edge within $\mathcal{F}_i$, meaning that there exists $A \in \mathcal{F}_i$ such that $|F \cap A| = t$. Consider the sunflower with kernel $F \cap A$ that was included in $\mathcal{S}_1$. This sunflower has at least $s - (k-t)$ petals that share no elements with $F$ outside $F \cap A$. These petals correspond to vertices in $G[\mathcal{S}_1]$ adjacent to $F$.

Thus, we can lower bound the number of edges that have at least one vertex in $\mathcal{S}_1$ as follows.  
\[ |\mathcal{S}_1| \cdot \frac{s-1}{2} + (s-(k-t))(|\mathcal{F}^*| - |\mathcal{S}_1|) \geqslant(s-(k-t)) \left(\ell - \frac{\alpha(G(n,k,t))}{c(k,d)} - s \binom{n}{t} \right).
\]

Next, we remove $\mathcal{S}_1$ from $\mathcal{F}$. In the same way we decompose $\ff \setminus \mathcal{S}_1$, find $\mathcal{S}_2$ and remove it, repeat the same step iteratively until the remaining family contains at most $\frac{\alpha(G(n,k,t))}{c(k,s)} = o(\ell)$ sets. Let $h$ be the number of iterations of the procedure. In each step we remove at most $s \cdot \binom{n}{t}$ vertices of $G[\ff]$, so 
\[
h \geqslant \left\lfloor \frac{\ell - \frac{\alpha(G(n,k,t))}{c(k,s)}}{s \cdot \binom{n}{t}} \right\rfloor \sim \frac{\ell \cdot t!}{s \cdot n^t}.
\]
At the $i$-th step of the procedure, we find at least
\[
(s-(k-t)) \left(\ell - \frac{\alpha(G(n,k,t))}{c(k,s)} - i \cdot s \binom{n}{t} \right)
\]
new edges in $G[\ff]$.

Thus, we get the bound
$$
\rho(\mathcal{F}) \geqslant (s-(k-t)) \sum_{i=1}^h \left( \ell - \frac{\alpha(G(n,k,t))}{c(k,d)} - i \cdot s \binom{n}{t} \right) \geqslant$$
$$(s-(k-t)) h \cdot \frac{\ell - \frac{\alpha(G(n,k,t))}{c(k,s)} - s \binom{n}{t}}{2} = (1+o(1)) \cdot \frac{l^2}{n^t} \cdot \frac{t!}{2} \cdot \frac{s-(k-t)}{s}.$$

Taking $s \to \infty$, this yields the desired bound.

\section{Proof of Theorem 5}\label{sec4}

We begin the proof with the harder case $r= o(n^{k-t-1})$.

\subsection{Case $r=o(n^{k-t-1})$}\label{sec41}

Consider an arbitrary family $\mathcal{F} \subset \binom{[n]}{k}$ satisfying $|\mathcal{F}| = \ell = \binom{n-t-1}{k-t-1} + r$. In the rest of the proof we will assume that $n$ is sufficiently large.

Let $f=f(n)$ be a very slowly increasing function with positive integer values and such that $f(n) \to \infty$ as $n \to \infty$. Note that, fucntion $f$ depends on $r(n)$ as well. For instance, $f = o\left(\ln\left(\frac{n^{k-t-1}}{r}\right)\right)$ will suit our purposes.

First, we need to obtain a decomposition of $\mathcal{F}$ as in the proof of Theorem~\ref{thm4}. Let $c(k) = c(k,2k)$ be the constant from Theorem~\ref{thm8}. We apply Theorem~\ref{thm8} with $s=2k$ to our family $\mathcal{F}$ and obtain $\mathcal{F}_1$. Then, similarly, apply Theorem 8 to $\ff_1' := \mathcal{F} \setminus \mathcal{F}_1$ and so on. We stop, if the number of the remaining sets is at most $\frac{1}{f(n)} \binom{n}{k-t-1}$. Let $m$ be the number of obtained families, and let $\mathcal{F}_{m+1} := \ff_m'$ be the family of remaining sets.  

This process yields a decomposition of the original family:
\[
\mathcal{F} = \mathcal{F}_1 \sqcup \mathcal{F}_2 \sqcup \ldots \sqcup \mathcal{F}_m \sqcup \mathcal{F}_{m+1}.
\]
Denote $\mathcal{F}^* = \mathcal{F} \setminus \mathcal{F}_{m+1}$. 

By our algorithm, $|\mathcal{F}^*| \geqslant \ell - \frac{1}{f(n)} \binom{n}{k-t-1} \sim \ell.$ Furthermore, for every $i\in[m]$ the inequality 
\begin{align}
|\mathcal{F}_i| \geqslant \frac{c(k)}{f(n)} \binom{n}{k-t-1}
\end{align}
holds. Thus, we get the bound $m =O (f)$.

For sufficiently large $n$, the inequality $|\ff_i| > \binom{n}{k-t-2}$ holds for each $i \in [m]$, and thus we can say that $\operatorname{rank}(\operatorname{Int}(\mathcal{F}_i)) \geqslant k-t-1$ by Claim~\ref{cla11}. This allows us to apply Lemma~\ref{lem9} to $\operatorname{Int}(\ff_i)$. Fix some $i \in [m]$. If condition (a) holds, then we say that $\mathcal{F}_i$ is of type 1, else we say that it is of type 2. The next claim implies that even one family of type 1 contain enough edges to guarantee the bound from the theorem. It is stated in a more general form, so we can apply it later in similar situations.

\begin{Claim}\label{cla12}
If $\cG \subset \binom{[n]}{k}$ is a family obtained as a result of an application of Theorem~\ref{thm8} with $s=2k$ and it is not $t$-avoiding, then 
$$\rho(\cG) \geqslant \frac{(k+t)|\cG|^2}{4k|\partial^{(t)}(\cG)|}.$$
\end{Claim}

\begin{proof}
The family $\cG$ is not $t$-avoiding, and so there exists $J \in \operatorname{Int}(\cG)$ such that $|J|=t.$ Denote
$\mathcal{X} = \{F_J : F \in \cG \}.$
Clearly, $|\mathcal{X}| \leqslant |\partial^{(t)}(\cG)|$ and $\sum_{X \in \mathcal{X}} |\cG[X]| = |\cG|$.

Fix an arbitrary $X \in \mathcal{X}$. By the condition (2) of Theorem~\ref{thm8}, $\cG(X)$ is $2k$-diverse. Consider any $F \in \cG[X]$.  We can lower bound the number of sets $A \in \cG$ such that $F \cap A = X$ as follows:
$$|\cG(X)| - \sum_{i \in F \setminus X} |\cG(X\cup\{i\})| \geqslant |\cG(X)| - \cfrac{k-t}{2k}|\cG(X)| = \cfrac{k+t}{2k}|\cG(X)|.$$
This expression is a lower bound of the degree of each vertex in the subgraph  $G[\cG[X]]$. The total number of edges thus satisfies 
$$\rho(\cG[X]) \geqslant  \cfrac{k+t}{4k}|\cG[X]|^2.$$

Finally, we sum up this bound over all $X \in \mathcal{X}$ and apply the Cauchy-Schwarz inequality. We obtain
$$\rho(\cG) \geqslant \sum_{X \in \mathcal{X}} \rho(\cG[X]) \geqslant \cfrac{k+t}{4k}|\cG[X]|^2 \geqslant \cfrac{k+t}{4k}\cdot \frac{(\sum_{X \in \mathcal{X}} |\cG[X]|)^2}{|\mathcal{X}|}$$
$$\geqslant \frac{(k+t)|\cG|^2}{4k|\partial^{(t)}(\cG)|}.$$

\end{proof}

If for some $i \in [m]$ the family $\ff_i$ is of type 1, then by Claim~\ref{cla11} we get
$$\rho(\mathcal{F}_i) \geqslant \frac{(k+t)|\ff_i|^2}{4k|\partial^{(t)}(\ff_i)|} \geqslant \frac{(k+t)|\ff_i|^2}{4k\binom{n}{t}} = \Omega \left(\frac{n^{2k-3t-2}}{f^2}\right).$$
Since $f^2=o(\frac{n^{k-t-1}}{r})$ we conclude that $\rho(\mathcal{F}_1) \gg rn^{k-2t-1}$ in this case. Let us clarify that by the notation $q(n) \gg p(n)$ we mean, here and henceforth, that $\lim _{n \rightarrow \infty} \frac{q(n)}{p(n)}=\infty.$

In what follows, we assume that for all $ i \in [m]$ the family $\mathcal{F}_{i}$ is of type 2. That is, all these families avoid intersection of size $t$. Let $M_i \in \binom{[k]}{t+1}$ be the center of the corresponding $\mathcal{F}_{i}$. For each $F \in \mathcal{F}_{i}$ denote $C(F) := F_{M_i}$.

Let $C_1, C_2, \ldots, C_p$ be all distinct sets $C(F)$ for $F \in \mathcal{F}^*$. For $i \in [p]$ define
$$\mathcal{G}_i = \{F \in \mathcal{F}^*: C(F) = C_i\}.$$

The next step of the proof is to select a small number of families $\mathcal{G}_i$ such that these families constitute a bulk of $\mathcal{F}^*$.

\begin{Claim}\label{cla13}
    \begin{align}
       \sum_{i=1}^p |\partial^{(t)}(\mathcal{G}_i(C_i))| \leqslant m \binom{n}{t}.
    \end{align}  
\end{Claim}

\begin{proof}

Note that
$$\sum_{j=1}^p |\partial^{(t)}(\mathcal{G}_j(C_j))| \leqslant \sum_{i=1}^m \sum_{j=1}^p |\partial^{(t)}((\mathcal{G}_j \cap \ff_i)(C_j))|.$$

Fix an arbitrary $i \in [m]$ and consider the family $\ff_i$ and the corresponding center $M_i$. We will show that $\partial^{(t)}((\mathcal{G}_j \cap \ff_i)(C_j))$ are disjoint for different $j \in [p]$. Then we have
$$ \sum_{j=1}^p|\partial^{(t)}((\mathcal{G}_j \cap \ff_i)(C_j))| \leqslant \binom{n}{t},$$
ans thus, summing up over $i \in [m]$, we get (2).

To show that the shadows are disjoint, we need to check that if $F_1, F_2 \in \mathcal{F}_i$ have different centers (i.e. $C(F_1) \neq C(F_2)$), then $F_1 \setminus C(F_1)$ and $F_2 \setminus C(F_2)$ share less than $t$ elements. Denote $A = F_1 \cap F_2$, $A_1 = C(F_1) \cap C(F_2)$ and $A_2 = A \setminus A_1$. Note that $\pi(A) \in \operatorname{Int}(\mathcal{F}_i)$ and $|A_1| \leqslant t$ holds, since $A_1 \subsetneq C(F_1)$ and $|C(F_1)|=t+1$. Clearly, we have $\pi(A_1) \subset M_i$ and $\pi(A_2) \cap M_i = \emptyset$. Our aim is to prove that $|A_2| < t$. 

Arguing indirectly, assume that $|A_2| \geqslant t$. Let $A_2'$ be any subset of $A_2$ such that $|A_2'|= t - |A_1|.$ By the condition (b) of the Lemma~\ref{lem9} we say that each superset of $M_i$ belongs to $\operatorname{Int}(\mathcal{F}_i)$, so $(M_i \cup \pi(A_2')) \in \operatorname{Int}(\mathcal{F}_i)$. Recall that $\operatorname{Int}(\mathcal{F}_i)$ is closed under intersection. Therefore, $(M_i \cup \pi(A_2')) \cap \pi(A) = \pi(A_1) \sqcup \pi(A_2')$ is in $\operatorname{Int}(\mathcal{F}_i)$ as well. However, $|\pi(A_1) \sqcup \pi(A_2')| = t$, which contradicts the fact that $\mathcal{F}_i$ is $t$-avoiding.  
\end{proof}

Assume that $|\partial^{(t)}(\mathcal{G}_i(C_i))| = \binom{x_i}{t},$ where $x_i$ are nonnegative real numbers, and that $x_1 \geqslant x_2 \geqslant \ldots \geqslant x_p.$

Denote $h=f^{t+1}$.
\begin{Claim}\label{cla14}
    $$\sum_{i=1}^h |\mathcal{G}_i| \geqslant (1 - o(1) )  \binom{n}{k-t-1}.$$
\end{Claim}

\begin{proof}

Since $x_1 \geqslant \ldots \geqslant x_p$, Claim~\ref{cla13} implies $ |\partial^{(t)}(\mathcal{G}_{h+1}(C_{h+1}))|  =  \binom{x_{h+1}}{t}  \leqslant \frac{m}{h+1} \binom{n}{t}.$ This inequality implies 
\begin{align}
x_{h+1} - t + 1 \leqslant\left(\frac{m}{h+1}\right)^{1/t} n = O\left(\frac{n}{f}\right).
\end{align}

By the Kruskal-Katona Theorem we obtain
$$|\mathcal{G}_i(C_i)| \leqslant\binom{x_i}{k-t-1}.$$

For $i > h+1$ we have $x_i \leqslant x_{h+1},$ and so $ \frac{\binom{x_{i}}{k-t-1}}{\binom{x_{i}}{t}} \leqslant \frac{\binom{x_{h+1}}{k-t-1}}{\binom{x_{h+1}}{t}}$ holds. Therefore, 
$$\sum_{i>h} |\mathcal{G}_i(C_i)| \leqslant\sum_{i>h} \binom{x_i}{k-t-1} = \sum_{i>h} \binom{x_i}{t} \frac{\binom{x_i}{k-t-1}}{\binom{x_i}{t}} \leqslant m \binom{n}{t} \frac{\binom{x_{h+1}}{k-t-1}}{\binom{x_{h+1}}{t}}$$

$$\leqslant m \binom{n}{t} \frac{t!x_{h+1}^{k-2t-1}}{(k-t-1)!}.$$

Inequality (2) implies $x_{h+1}^{k-2t-1} = o(\frac{n^{k-2t-1}}{f}).$ Hence,
$$ m \binom{n}{t} \frac{t!x_{h+1}^{k-2t-1}}{(k-t-1)!} = o(n^{k-t-1}).$$

$$\sum_{i=1}^h|\mathcal{G}_i| = |\ff^*| - \sum_{i>h} |\mathcal{G}_i(C_i)| \geqslant |\ff^*| - m \binom{n}{t} \frac{\binom{x_{h+1}}{k-t-1}}{\binom{x_{h+1}}{t}} = (1-o(1))\binom{n}{k-t-1}.$$

\end{proof}

Now we consider only the families $\mathcal{G}_i$ for $i \in [h]$ that form a bulk of $\mathcal{F}$. Next, we will apply a peeling procedure to these families. 

Denote $W = \bigcup_{i=1}^h C_i$. Note that $|W| \leqslant h(t+1) = o(n^\varepsilon)$ for any constant $\varepsilon>0$. Define another slowly increasing function $\tau(n) = hf = f^{t+2}.$

\begin{Claim}\label{peel}  For all $i \in [h]$ there exist families $\mathcal{G}_i' \subset  \mathcal{G}_i$ such that the following properties hold.\\
(1) $\sum_{i=1}^h |\mathcal{G}_i| - \sum_{i=1}^h |\mathcal{G}_i'| = o(n^{k-t-1}).$\\
(2) For each $i \in [h]$ and $G \in \mathcal{G}_i'$, $G \cap W = C_i.$\\
(3) For each $i \in [h]$ and $X \in \partial^{(t)}(\cG_i'(C_i)),$ $|\cG_i'(C_i \cup X)| \geqslant \frac{n^{k-2t-1}}{\tau}.$\\
(4) For each $i \in [h]$ and $X \subset [n] \setminus W$ with $|X| \leqslant t$, the family $\mathcal{G}_i'(C_i \cup X)$ is $\tau$-diverse.

\end{Claim}

\begin{proof}
First, for $i \in [h]$ define
$$\mathcal{G}_i' := \{G \in \mathcal{G}_i: G \cap W = C_i\}.$$
Clearly, every set in $\mathcal{G}_i \setminus \mathcal{G}_i'$ has at least $t+2$ elements in $W$, and so $|\cG_i \setminus \cG_i'| \leqslant  \binom{|W|}{t+2}\binom{n}{k-t-2}.$ We have $ \sum_{i=1}^h   |\cG_i \setminus \cG_i'| \leqslant   h\binom{|W|}{t+2}\binom{n}{k-t-2}=o(n^{k-t-1})$ sets. Thus, the families $\mathcal{G}_i'$ satisfy conditions (1) and (2). 

Next, we run the following peeling procedure. If some $\cG_i'$ and $X \in \binom{[n] \setminus W}{\leq t}$ violate condition (3) or (4), then put $\cG_i' := \cG_i' \setminus \cG_i'[X]$. If such $i$ and $X$ do not exist, then the process stops. Clearly, once the process stops, properties (3) and (4), as well as (2), are satisfied. Let us upper bound the number of removed sets.

Assume that in our process we find some $X \in \binom{[n] \setminus W}{\leq t}$ and $i \in [h]$ such that $\mathcal{G}_i'(C_i \cup X)$ is not $\tau$-diverse. Then there exists $j \in [n] \setminus (C_i \cup X)$ such that $|\mathcal{G}_i'(C_i \cup X \cup \{j\})| \geqslant \frac{1}{\tau} |\mathcal{G}_i'(C_i \cup X)| = \frac{1}{\tau} |\mathcal{G}_i'(X)|.$ Clearly, we have $|\mathcal{G}_i'(C_i \cup X \cup \{j\})| \leqslant \binom{n}{k-t-2-|X|}$. Therefore, every such deleted family $\mathcal{G}_i'[X]$ satisfies $|\mathcal{G}_i'[X]|\leqslant \tau \cdot \binom{n}{k-t-2-|X|}$.

If we found some $i \in [h]$ and $X \in \binom{[n] \setminus W}{t}$ that violate condition (3), then we, obviously, remove at most $\frac{n^{k-2t-1}}{\tau}$ sets in this step.

It is clear that we can make at most one deletion for each $i \in [h]$ and $X \in \binom{[n] \setminus W}{\leq t}$. Thus, the total number of removed sets does not exceed 
$$h\tau \sum_{j=1}^t \binom{n}{j}\binom{n}{k-t-2-j} + h\binom{n}{t}\frac{n^{k-2t-1}}{\tau}  = o(n^{k-t-1}),$$
which means that property (1) still holds at the end of the procedure.
\end{proof}

Apply Claim~\ref{peel} to our families $\mathcal{G}_i$ and get families $\mathcal{G}_i'$ with "regularized" properties. Denote $\mathcal{G} = \bigcup_{i \in [h]} \mathcal{G}_i'.$

For any families $\cA$ and $\bb$ we say that $\rho(\cA,\bb)$ is the number of unordered pairs $A \in \cA$ and $B \in \bb$ such that $|A \cap B|=t$. Fix distinct $i, j \in [h]$ and assume that $|C_i \cap C_j| = t-x$ for some $x \geqslant 0$. Our aim is to lower bound $\rho(\cG_i', \cG_j')$.

Fix an arbitrary $G \in \mathcal{G}_i'$ and some $X \subset G \setminus C_i$ such that $|X|=x$. Note that, by Claim~\ref{peel} (2), $X \subset [n] \setminus W$ and for each $H \in \cG_j'$, we have $H \cap G = (C_i \cap C_j) \sqcup ((G \setminus C_i) \cap (H \setminus C_j)).$ By Claim~\ref{peel} (4), applied for $\cG_j'$, we have at least
\begin{align*}
    |\mathcal{G}_j'(C_j \cup X)| - \sum_{a \in G \setminus (C_i \cup X)} |\mathcal{G}_j'( C_j \cup X \cap \{a\})| \\\geqslant|\mathcal{G}_j'(C_j \cup X)| - \cfrac{k}{\tau}|\mathcal{G}_j'(C_j \cup X)|
    = \cfrac{\tau - k}{\tau} \cdot |\mathcal{G}_j'(C_j \cup X)|    
\end{align*}
sets $F \in \mathcal{G}_j'$ such that $F \cap G = X \cup (C_i \cap C_j)$, so $|F \cap G|=t$ in this case. 

Summing up this inequality over all $G$ and $X$, we obtain
$$
\rho(\mathcal{G}_i', \mathcal{G}_j') \geqslant\sum_{G \in \mathcal{G}_i'} \sum_{\substack{X \subset G \setminus C_i,\\ |X|=x}} \frac{\tau - k}{\tau} \cdot |\mathcal{G}_j'(X)| = \frac{\tau - k}{\tau}  \sum_{\substack{X \subset [n] \setminus W,\\ |X|=x}} |\mathcal{G}_i'(X)||\mathcal{G}_j'(X)|.
$$

Clearly, inequality
$$\sum_{\substack{X \subset [n] \setminus W,\\ |X|=x}} |\mathcal{G}_i'(X)||\mathcal{G}_j'(X)| \geqslant \sum_{\substack{X \subset [n] \setminus W,\\ |X|=t}} |\mathcal{G}_i'(X)||\mathcal{G}_j'(X)|$$
holds. So we have
$$\rho(\cG_i', \cG_j') \geqslant \frac{\tau-k}{\tau} \sum_{\substack{X \subset [n] \setminus W,\\ |X|=t}} |\mathcal{G}_i'(X)||\mathcal{G}_j'(X)|$$
for every distinct $i,j \in [h]$. In a view of Claim~\ref{peel} (3), for every common set $X\in (\partial^{(t)}(\cG_i') \cap \partial^{(t)}(\cG_j'))$ we get at least
\begin{align}
    \frac{\tau-k}{\tau}\left(\frac{n^{k-2t-1}}{\tau}\right)^2    
\end{align}
edges.

If $\sum_{i=1}^h |\partial^{(t)}(\cG_i')| \geqslant \binom{n}{t} + \frac{r\tau^3}{n^{k-2t-1}}$, then we obtain $\sum_{i<j, i,j \in [h]} |(\partial^{(t)}(\cG_i') \cap \partial^{(t)}(\cG_j'))| \geqslant \frac{r\tau^3}{n^{k-2t-1}}$. Together with (4), it implies
$$\rho(\cG) \geqslant \frac{r\tau^3}{n^{k-2t-1}} \cdot \frac{\tau-k}{\tau}\left(\frac{n^{k-2t-1}}{\tau}\right)^2 \gg rn^{k-2t-1},$$
and we are done. So we can assume that 
\begin{equation}
    \sum_{i=1}^h |\partial^{(t)}(\cG_i')| < \binom{n}{t} +  \frac{r\tau^3}{n^{k-2t-1}} \sim \binom{n}{t},    
\end{equation}
where the last equality is since $\tau^3 = o(\frac{n^{k-2t-1}}{r})$ by our choice of $\tau$.

Without loss of generality, assume that for $i \in [h]$ equality $|\cG_i'|=\binom{y_i}{k-t-1}$ holds, where $y_1 \geqslant y_2 \geqslant \ldots \geqslant y_h$. Our next goal is to prove that $y_1 = (1+o(1))n$, or, put differently, that $|\cG_1'| \sim \ell.$ Roughly speaking, it means that almost all of our sets must contain the same $(t+1)$-element center.

By the Kruskal-Katona Theorem and the inequality (5), we have $$ \sum_{i=1}^h \binom{y_i}{t} \leqslant  \sum_{i=1}^h |\partial^{(t)}(\cG_i')| \leqslant \binom{n}{t} + o(n^t).$$
 
By Claim~\ref{cla14} we obtain

$$\binom{n}{k-t-1} + o(n^{k-t-1}) \leqslant \sum_{i=1}^h |\cG_i'| =\sum_{i=1}^h \frac{\binom{y_i}{k-t-1}\binom{y_i}{t}}{\binom{y_i}{t}} $$ 
$$\leqslant \sum_{i=1}^h \binom{y_i}{t} \frac{\binom{y_1}{k-t-1}}{\binom{y_1}{t}} \leqslant 
\left(\binom{n}{t} + o(n^t)\right) \frac{\binom{y_1}{k-t-1}}{\binom{y_1}{t}}.$$

It implies $y_1 \sim n$ and, consequently, $|\cG_1'| = \binom{y_1}{k-t-1} \sim \binom{n}{k-t-1}$. 

Now we can return to our original family $\ff$. Without loss of generality, we say that $C_1=[t+1].$ Define
$$\cA = \{F \in \ff : [t+1] \subset F \}, \quad \bb = \ff \setminus \cA.$$
Since $\cG_1' \subset \cA$, we have $|\bb| = o(n^{k-t-1})$. Thus, almost all sets in $\ff$ form an independent set $\cA$ in graph $G[\ff]$. Further, we will prove that there exists a sufficiently large number of edges either inside $\bb$, or between $\cA$ and $\bb$.

Note that $|\cA| \leqslant \binom{n-t-1}{k-t-1}$ implies $|\bb| \geqslant r.$ Denote
$$\cA' = \left\{F \notin \ff \mid [t+1] \subset F, F \in \binom{[n]}{k} \right\}.$$
We have $|\cA'| = |\bb| -r.$

Consider an arbitrary set $B \in \bb$. If $B \cap [t+1] = \emptyset$, then it has exactly $\binom{k}{t}\binom{n-k-t-1}{k-2t-1}$ neighbors in the family $\cA \sqcup \cA'$. If $B \cap [t+1] \neq \emptyset$, then the number of neighbors is at least $cn^{k-2t}$ for some constant $c>0$. Put $d=\binom{n-k-t-1}{k-2t-1}.$ Let $\bb'\subset \bb$ be the family of all sets in $\bb$, which have at most $\frac{d}{2}$ neighbors in $\cA$. If $|\bb'| < \frac{|\bb|}{2}$, then we have at least $\frac{d}{2} \cdot \frac{|B|}{2} \geqslant \frac{rd}{4}$ edges induced between $\bb$ and $\cA$, and we are done. So we can assume that $|\bb| \leqslant 2|\bb'|$.

\begin{Claim}\label{cla16}
    $$|\bb| \geqslant \frac{|\partial^{(t)}(\mathcal{B'})|d}{2\binom{k}{t}}.$$
\end{Claim}

\begin{proof}
    Consider an arbitrary $D \in \partial^{(t)}(\mathcal{B'})$. Fix some set $B \in \mathcal{B'}$ such that $D \subset B'$. There exist at least $d=\binom{n-k-t-1}{k-2t-1}$ sets $A \in (\cA \sqcup \cA')$ such that $A \cap B = D$. Since $B$ has at most $\frac{d}{2}$ neighbors in $\cA$, we obtain $|\cA'[D]| \geqslant \frac{d}{2}$. Let us double count this bound over all $D \in \partial^{(t)}(\mathcal{B'})$. We have
$$ |\mathcal{A}'| = \frac{\sum_{D \in \binom{[n]}{t}}|\cA'[D]|}{\binom{k}{t}} \geqslant \frac{\sum_{D \in \partial^{(t)}(\mathcal{B'})}|\cA'(D)|}{\binom{k}{t}} \geqslant \frac{|\partial^{(t)}(\mathcal{B'})|d/2}{\binom{k}{t}}.$$
Note that $|\bb| \geqslant |\cA'|$ and we are done.
\end{proof}

Let us apply Theorem~\ref{thm8} to the family $\bb'$ with $c(k,2k)$ and obtain a family $\bb'' \subset \bb'$, where $|\bb''| \geqslant c(k,2k)|\bb'|$. We will distinguish two cases, depending on whether $\bb''$ is $t$-avoiding or not.

First, assume that $\bb''$ is $t$-avoiding. In this case, Theorem 10 implies that there exists a constant $c$ such that $|\bb''| \leqslant  c|\partial^{(k-t-1)}(\bb'')|.$ Take $x$ such that $|\partial^{(k-t-1)}(\mathcal{B''})| = \binom{x}{k-t-1}.$ Since $|\bb'| =o(n^{k-t-1})$ we must have  $x=o(n)$. By the Kruskal-Katona Theorem, $|\partial^{(t)}(\mathcal{B''})| = |\partial^{(t)}(\partial^{(k-t-1)}(\mathcal{B''}))| \geqslant \binom{x}{t}$ holds.

We have $|\bb| \leqslant 2|\bb'| \leqslant c_1 \binom{x}{k-t-1}$ for some constant $c_1$. Combining this bound with Claim~\ref{cla16}, we get 
$$ c_1 \binom{x}{k-t-1} \geqslant |\bb| \geqslant \frac{|\partial^{(t)}(\mathcal{B''})|d}{2\binom{k}{t}} \geqslant  \frac{\binom{x}{t}d}{2\binom{k}{t}} =\frac{\binom{x}{t}\binom{n-k-t-1}{k-2t-1}}{2\binom{k}{t}}.$$
But $x^{k-t-1}=o(x^tn^{k-2t-1})$, so we arrive at a contradiction for sufficiently large $n$.

Now we can assume that the family $\bb''$ is not $t$-avoiding. In this case we can use Claim~\ref{cla12} and find a lot of edges inside $\bb''$. Indeed, Claim~\ref{cla12} gurantees  $\rho(\bb'') \geqslant \frac{c_2|\bb''|^2}{|\partial^{(t)}(\bb'')|}$ for some constant $c_2 > 0$. Clearly, $|\partial^{(t)}(\bb'')| \leqslant |\partial^{(t)}(\bb')|$, so by Claim~\ref{cla16} we have $$|\partial^{(t)}(\bb'')| \leqslant \frac{2\binom{k}{t}|\bb|}{d}.$$
Finally, we obtain
$$\rho(\bb'') \geqslant \frac{c_2|\bb''|^2}{|\partial^{(t)}(\bb'')|} \geqslant  \frac{c_2|\bb''|^2d}{2\binom{k}{t}|\bb|} \geqslant c_3|\bb|d$$
for some constant $c_3>0$. We have $|\bb| \geqslant r$ and $d=\Theta(n^{k-2t-1})$. The proof is completed.

\subsection{Two extremal constructions}\label{sec42}
We want to highlight that there are at least two different types of constructions with $\Theta (rn^{k-t-1})$ edges. This is one of the reason for the complications in the previous section. Put $\cA_1 = \{F : F \in \binom{[n]}{k}, [t+1] \subset F \}$ and $\cA_2 = \{F : F \in \binom{[n]}{k}, [t+2,2t+2] \subset F, [t+1] \cap F = \emptyset \}$. We say that a family of sets is a \textit{$s$-star}, if all sets in this family contain common $s$ elements (the set of these elements is called the center). Thus, $\cA_1$ and $\cA_2$ are $(t+1)$-stars with disjoint centers. The most natural construction with few edges consists of $\cA_1$ and an arbitrary subfamily of $\cA_2$ with $r$ sets. In this case, the corresponding subgraph contains exactly $r \cdot \binom{k}{t}\binom{n-k-t-1}{k-2t-1}$ edges. This example corresponds to the case in our proof, where we found the required number of edges between $\cA$ and $\bb$. Clearly, we can extend this example to $d=\Theta(n^{k-t-1})$. Simply take a constant number of entire $(t+1)$-stars with disjoint centers and a part of another one such $(t+1)$-star.

Now let us describe another type of construction. The family $\cA_1$ will form a bulk of this construction as well. Initially, find the smallest positive integer $x=x(n)$ such that 
\begin{align}
    \binom{x}{k} \geqslant r + \binom{x}{t}\binom{n}{k-2t-1}.    
\end{align}
Consider the families 
$$\cA_3= \binom{[t+2,t+1+x]}{k},$$ $$\cA_4=\{F : F \in \cG_1, |F \cap [t+2,t+1+x]| < t \}.$$
Clearly, $|\cA_3|=\binom{x}{k}$ and $|\cA_4| \geqslant \binom{n-t-1}{k-t-1} - \binom{x}{t}\binom{n}{k-2t-1}$. It is easy to see that $\rho(\cA_3) = \frac{1}{2} \binom{x}{k} \binom{k}{t} \binom{x-k}{k-t}$. Construct a family that contains $\cA_4$, as well as arbitrary $\ell-|\cA_4|$ sets from $\cA_3$. Thus, all induced edges are inside $\cA_3$. Therefore, the number of edges in our subgraph does not exceed $\frac{1}{2} \binom{x}{k} \binom{k}{t} \binom{x-k}{k-t}$. Suppose that $r \sim n^{\frac{k(k-2t-1)}{k-t}}$. We have $r=o(n^{k-t-1})$, since $k(k-2t-1) < (k-t)(k-t-1)$. In this case functions, the $\binom{x}{k}, r$ and $\binom{x}{t}\binom{n}{k-2t-1}$ all have the same order of growth, see (6). We obtain the required magnitude of the number of edges, since $ \binom{x}{k} \binom{x-k}{k-t} = \Theta(rn^{k-2t-1})$. This example corresponds to the case of our proof, when we found many edges inside $\bb$.

\subsection{Case $r=\Theta(n^{k-t-1})$}\label{sec43}

We now adapt the proof from Section~\ref{sec41} to the case of $r = \Theta(n^{k-t-1})$ by modifying the relevant parameters. 

Take a constant $\varepsilon > 0$ such that $\ell \geqslant (1 + \varepsilon)\binom{n-t-1}{k-t-1}.$ The first step of the proof is applying Theorem~\ref{thm8}. We stop the procedure that finds the families $\ff_i$, when the number of remaining sets is less than $\frac{\varepsilon}{2}\binom{n}{k-t-1}$. In this case, we get
$|\ff_i| \geqslant \frac{c(k)\varepsilon}{2} \binom{n}{k-t-1}$ for each $i \in [m]$. Therefore, we can upper bound $m$ by some constant. If some $\ff_i$ is of type 1, then by Claim~\ref{cla12} we have $\Omega(\frac{\ell^2}{n^t})$ edges inside $\ff_i$. Thus, we again have families of type 2 only.

The next essential difference is the choice of $h$ for Claim~\ref{cla14}. We weaken the inequality and require only
$$\sum_{i=1}^h |\cG_i| \geqslant \left(1+\frac{\varepsilon}{4}\right)\binom{n}{k-t-1}.$$
In this way we can upper bound $h$ by a sufficiently large constant. 

For Claim~\ref{peel} we can put $\tau(n)=\ln n$, which is enough to fulfill the conditions. We get a constant number of families $$\cG = \cG_1' \sqcup \cG_2' \sqcup \ldots \sqcup \cG_h',$$
satisfying the conditions (2)-(4) of Claim~\ref{peel} and such that $\sum_{i=1}^h|\cG_i'| \geqslant \left(1+\frac{\varepsilon}{4}\right)\binom{n}{k-t-1} - o(n^{k-t-1})$. As in previous proof, we have
\begin{align}\label{eq7}
\rho(\mathcal{G}_i' \sqcup \mathcal{G}_j') \geqslant \frac{\tau-k}{\tau} \sum_{\substack{X \subset [n] \setminus W,\\ |X|=t}} |\mathcal{G}_i'(X)||\mathcal{G}_j'(X)|.
\end{align}
From this point we can deviate from the previous proof and find a sufficient number of edges by double counting.

If we combine the bound (7) for all $i \neq j$ and $i,j \in [h]$, then we obtain

    $$\rho(\mathcal{F}) \geqslant \frac{\tau - k}{2\tau} \sum_{\substack{X \subset [n] \setminus W,\\ |X|=t}}\left(\left(\sum_{i \in [h]}|\mathcal{G}_i'(X)|\right)^2 - \sum_{i \in [h]}|\mathcal{G}_i'(X)|^2\right)$$

\begin{align}
=\frac{\tau - k}{2\tau} \sum_{\substack{X \subset [n] \setminus W,\\ |X|=t}}\left(|\cG(X)|^2 - \sum_{i \in [h]}|\mathcal{G}_i'(X)|^2\right).    
\end{align}

Every set in $\cG$ has exactly $k-t-1$ elements in $[n] \setminus W$, so
$$\sum_{\substack{X \subset [n] \setminus W,\\ |X|=t}} |\mathcal{G}(X)| = \binom{k-t-1}{t} |\cG|$$
holds.

Therefore, we have
$$\sum_{\substack{X \subset [n] \setminus W,\\ |X|=t}} |\mathcal{G}(X)|^2 \geqslant \cfrac{\left(\binom{k-t-1}{t} |\mathcal{G}|\right)^2}{\binom{n}{t}}$$
by the Cauchy-Schwarz inequality.

Note that for all $X \subset \binom{[n] \setminus W}{t}$ and $i \in [h]$ we have $|\mathcal{G}_i'(X)| \leqslant \binom{n}{k-2t-1}$. Thus, by convexity we obtain
$$\sum_{\substack{X \subset [n] \setminus W,\\ |X|=t}} \sum_{i \in [h]}|\mathcal{G}_i'(X)|^2 \leqslant \binom{n}{k-2t-1}  \binom{k-t-1}{t} |\mathcal{G}|.$$

Now we combine these bounds and get
$$\sum_{\substack{X \subset [n] \setminus W,\\ |X|=t}}\left(|\cG(X)|^2 - \sum_{i \in [h]}|\mathcal{G}_i'(X)|^2\right)$$
$$\geqslant \cfrac{\left(\binom{k-t-1}{t} |\mathcal{G}|\right)^2}{\binom{n}{t}} - \binom{n}{k-2t-1}  \binom{k-t-1}{t} |\mathcal{G}|$$

$$=  \binom{k-t-1}{t} |\mathcal{G}|\left(\cfrac{(k-t-1)!|\mathcal{G}|}{(k-2t-1)!n^t} - \cfrac{n^{k-2t-1}}{(k-2t-1)!} + o(n^{k-2t-1})\right)$$

$$\geqslant \binom{k-t-1}{t} |\mathcal{G}|\left(\cfrac{\varepsilon n^{k-2t-1}}{4(k-2t-1)!} + o(n^{k-2t-1})\right) = \Theta(n^{2k-3t-2}).$$

Substitute this into (8). Since $\frac{\tau - k}{\tau} \sim 1$, we get the claimed bound.

\section{Exact values of $\rho(l)$}\label{sec5}

\subsection{Proof of Theorem 6}\label{sec51}

Recall the notation $d =\binom{n-k-t-1}{k-2t-1}.$ Consider an arbitrary family $\mathcal{F} \subset\binom{[n]}{k}$ such that $|\mathcal{F}| = \ell$. Arguing indirectly, assume that $\rho(\mathcal{F}) < r\binom{k}{t}d.$ In the rest of the proof, we will assume that $n$ is sufficiently large. Running the proof of Theorem~\ref{thm5}, we can derive that there exists a set $C \in \binom{[n]}{t+1}$ such that $|\mathcal{F}_0(C)| \sim \ell$, otherwise $\rho(\ff) \gg rd$.

Let us use some notation from the last part of Section~\ref{sec41}. Denote $\mathcal{B} = \mathcal{F} \setminus \mathcal{F}[C]$, $\mathcal{A} = \mathcal{F}[C]$, $\mathcal{K} = \{A \in \binom{[n]}{k}: C \subset A \}$ and $\mathcal{A}' = \mathcal{K} \setminus \mathcal{A}$.

For each set $B \in \binom{n}{k} \setminus \mathcal{K}$ define 
$$\mathcal{N}_B = \{K \in \mathcal{K} : |K \cap B| = t \}.$$

If $B \cap C = \emptyset$, then
$|\mathcal{N}_B| = \binom{k}{t}d$, else $|\mathcal{N}_B| \gg n^{k-2t-1}$.

Since $|\cA'| = |\mathcal{B}| - r$,  for every $B \in \mathcal{B}$ there exist at least $\binom{k}{t}d - (|\mathcal{B}| - r )$ adjacent vertices in $\cA$. Thus, $$ r\binom{k}{t}d > \rho(\mathcal{F}) \geq |\mathcal{B}|\left(\binom{k}{t}d+r - |\mathcal{B}|\right).$$

Since $|\bb| \geqslant r$, the inequality valid only if $|\mathcal{B}| > \binom{k}{t}d \gg r.$ Just as in the proof of Theorem~\ref{thm5}, we denote $\bb'= \{ B : |\mathcal{N}_B \cap \cA| \leqslant \frac{d}{2}, B \in \bb  \}$. We have less than $r\binom{k}{t}d$ edges, so  $|\bb| < |\bb'| +2\binom{k}{t}r$ holds in our case. In particular, it means that $\bb'$ is not empty, and, moreover, forms a bulk of $\bb$ (for sufficiently large $n$).

Recall the corresponding step in the previous proof. Let us consider a subfamily $\bb'' \subset \bb'$ given by Theorem~\ref{thm8}. We consider two cases: whether the family $\bb''$ is $t$-avoiding or not.

If $\bb''$ is $t$-avoiding, then we arrive at the same contradiction as in the previous proof. If it is not $t$-avoiding, then we get
$$\rho(\bb'') \geqslant c_3 |\bb|d \gg rd.$$
Thus, we also arrive at a contradiction.

\subsection{One more construction}\label{sec52}

Let us construct a sequence of families $\ff \subset \binom{[n]}{k}$ such that $|\ff| = \binom{n-t-1}{k-t-1} + r(n)$, $r = \Theta(n^{k-2t-1})$ and $\rho(\ff) < r\binom{k}{t}\binom{n-k-t-1}{k-2t-1}=r\binom{k}{t}d$. Thus, we show that the order of growth of $r(n)$ in Theorem~\ref{thm6} can not be improved.

Take any function $r(n)$ such that $r(n) > \binom{n-2t-1}{k-2t-1}\left(\binom{k}{t} - 1\right)$ and $r(n) = o(n^{k-t-1})$.

Consider the family $\cG_1 = \{G \in \binom{[n]}{k} : [t+1] \subset G, [t+2, 2t+1] \nsubseteq G \}$. Clearly, we get $|\cG_1| = \binom{n-t-1}{k-t-1} - \binom{n-2t-1}{k-2t-1}$. Also, denote $\cG_2 = \{G \in \binom{[n]}{k} : [t+1] \cap G = \emptyset, [t+2, 2t+2] \subset G \}$ and choose an arbitrary $\cG_2' \subset \cG_2$ such that $|\cG_2'| = \binom{n-2t-1}{k-2t-1} + r(n)$. Indeed, it is possible, since $|\cG_2| = \Theta(n^{k-t-1})$.

Our example is the family $\ff = \cG_1 \sqcup \cG_2'$. Note that $\cG_1$ and $\cG_2'$ correspond to the independent sets in the graph $G(n,k,t)$. Therefore, we only need to count the edges between these families.

For each set $G \in \cG_2$, the family $\cG_1$ does not contain sets that intersect $G$ exactly in $[t+2,2t+1]$. Thus, each set $G \in \cG_2$ has exactly $$\left(\binom{k}{t} - 1\right)\binom{n-k-t-1}{k-2t-1} = \left(\binom{k}{t} - 1\right)d$$
edges to $\cG_1$. Hence, we see that
$$\rho(\ff) = \left(r + \binom{n-2t-1}{k-2t-1}\right)\left(\binom{k}{t} - 1\right)d < r\binom{k}{t}d,$$
since $r > \binom{n-2t-1}{k-2t-1}\left(\binom{k}{t} - 1\right)$.


\end{document}